\documentclass[12pt,twoside]{article}
\title{ Rings with finite $n$-Weak injective dimension and $(n,k)$-Weak cotorsion  modules}
\date{}
\author{}
\usepackage{graphicx,fancyhdr,fancybox,rotating,xcolor,appendix}
\usepackage{ulem,hhline,dsfont,pifont,multicol,lmodern}
\usepackage{amsthm,amsbsy,geometry,times,pst-node}
\usepackage{amsmath,tikz-cd}

\usepackage{lipsum}
\usepackage{euscript}
\usepackage{hyperref}
\usepackage{hyperref}
\usepackage{hyperref}

\usepackage{textcomp}
\usepackage{pdfpages}
\usepackage{stmaryrd}
\usepackage[all]{xy}
\usepackage{amsfonts}
\usepackage{amsmath}
\usepackage{amssymb}
\usepackage{latexsym}
\usepackage{mathrsfs}

\newtheorem{thm}{Theorem}[section]
 \newtheorem{cor}[thm]{Corollary}
 \newtheorem{lem}[thm]{Lemma}
 \newtheorem{prop}[thm]{Proposition}
 \newtheorem{Def}[thm]{Definition}
\newtheorem{rem}[thm]{Remark}
 \newtheorem{ex}[thm]{Example}

\newcommand{\X}{\rm \mathscr{X}}

\def\Ext{{\rm Ext}}
\def\Tor{{\rm Tor}}
\def\Hom{{\rm Hom}}
\def\Im{{\rm Im}}
\def\Ker{{\rm Ker}}

\newcommand{\F}{{\cal F}}

\def\Mod{{\rm Mod}}
\def\Tor{{\rm Tor}}
\begin{document}

\thispagestyle{empty}

\maketitle \vspace*{-1.5cm}
\begin{center}{\large\bf Mostafa Amini$^{1,a}$, Houda Amzil$^{2,b}$ and  Driss   Bennis$^{2,c}$ }
\bigskip

\small{1. Department of Mathematics, Faculty of Sciences, Payame Noor University, Tehran, Iran.\\
2. Department of Mathematics, Faculty of Sciences, Mohammed V University in Rabat, Morocco.\\
$\mathbf{a.}$ amini.pnu1356@gmail.com \\
$\mathbf{b.}$ houda.amzil@um5r.ac.ma; ha015@inlumine.ual.es \\
$\mathbf{c.}$  driss.bennis@um5.ac.ma; driss$\_$bennis@hotmail.com \\}

\small{$\mathbf{}$  
}
\end{center}

\bigskip

\noindent{\large\bf Abstract.} Let $R$ be a  ring and $n, k$ be two non-negative integers. As an extension of several known notions, we introduce and study $(n,k)$-weak cotorsion modules using the class of right $R$-modules with $n$-weak flat dimensions at most $k$. Various examples and applications are also given. \bigskip

\medskip

\bigskip\bigskip
%

\section{Introduction  }

Throughout this paper, $R$ will denote an associative (not necessarily commutative) ring with identity.
As usual, we denote by $R$-$\Mod$ and $\Mod$-$R$ the category of left $R$-modules and right R-modules, respectively. The class of all right (resp. left) $R$-modules of flat dimension at most a non-negative integer $k$ is denoted by $\mathscr{F}_{ k}(R^{op})$ (resp. $\mathscr{F}_{ k}(R)$) and for a module $M$, $E(M)$ denotes its injective envelope.

In 1984, Enochs introduced the concept of cotorsion modules as a generalization of cotorsion abelian groups \cite{EJ}. It differs from  Matlis' definition  in  \cite{ENA} whose  concern  was  with  domains. Since then, the investigation of these modules have become a vigorously active area of research. We refer the reader to \cite{GA,EJ,LM} for background on cotorsion modules. In 2007,  Mao and Ding in \cite{LM3}  introduced the concept of $k$-cotorsion left modules as a new classification of cotorsion modules by using the class $\mathscr{F}_{ k}(R)$.

 In 2015, Gao and Wang in \cite{Z.W}  introduced the concept  of weak injective and weak flat modules by using  super finitely presented modules instead of finitely presented modules. Then, they investigated the properties of $R$-modules with weak injective and weak flat dimension at most a non-negative integer $k$. Also, Gao and Huang in \cite{Z.H}, investigated the right  $\mathcal{WI}$-dimension of modules in terms of left derived functors of Hom and the left  $\mathcal{WI}$-resolutions of modules, where  $\mathcal{WI}$ is the class of weak injective $R$-modules. This process continues in 2018 when Zhao in \cite{.NG}  introduced another relative derived functor in terms of left or right $\mathcal{WF}$-resolutions to compute the weak flat dimension of modules,  where  $\mathcal{WF}$ is the class of weak flat $R$-modules. Then, they investigated homological properties of modules with finite weak injective and weak flat dimensions. Let $\mathcal{WI}_{ k}(R)$ and $\mathcal{WF}_{ k}(R^{op})$ be, respectively, the classes of left and right $R$-modules with weak
injective and weak flat dimensions at most $k$, respectively.
In 2018, Selvaraj and Prabakaran in \cite{B2B}, introduced a particular case of $k$-cotorsion modules and called them $k$-weak cotorsion modules by using the class $\mathcal{WF}_{ k}(R^{op})$  instead of  $\mathscr{F}_{ k}(R^{op})$. 

Recently in \cite{AB}, the notion  of $n$-super finitely presented left $R$-modules for any non-negative integer $n$ has been introduced as a particular case of super finitely presented. Then, by using $n$-super finitely presented left $R$-modules, a natural extension of the notions of weak injective and weak flat modules have been introduced and studied. They are called the $n$-weak injective and $n$-weak flat modules, respectively. Furthermore, some homological aspects of modules with finite $n$-weak injective and $n$-weak flat dimensions have been also investigated in \cite{AB}. Throughout,  $\mathcal{WI}^n(R)$, $\mathcal{WF}^n(R^{op})$, $\mathcal{WI}_{ k}^n(R)$ and $\mathcal{WF}_{ k}^n(R^{op})$ will denote respectively the classes of $n$-weak injective left $R$-modules,  $n$-weak flat right $R$-modules,  left $R$-modules with $n$-weak
injective  dimensions less than or equal to $k$ and right $R$-modules with $n$-weak flat dimensions less than or equal to $k$.

In this paper, a natural extension of the notion of weak cotorsion modules is introduced and studied which is called $(n,k)$-weak cotorsion  modules (see Definition \ref{Def-nk-weak-cot}). Every $(n,k)$-weak cotorsion module is $k$-weak cotorsion, $k$-cotorsion and cotorsion, and $(0,k)$-weak cotorsion modules are exactly $k$-weak cotorsion modules.

The paper is organized as follows:

In Sec. 2, some fundamental concepts and some preliminary results are stated. Namely, we recall the definitions of $n$-super finitely presented, $n$-weak injective  left $R$-modules, and $n$-weak flat  right  $R$-modules introduced in \cite{AB} (see Definitions \ref{Def-nsfp} and \ref{Def-nwi-nwf}). 

In Sec. 3, we introduce $(n,k)$-weak cotorsion modules. Then, examples are given to show that for $n>m\geq 0$, $(m,k)$-weak cotorsion $R$-modules need not be $(n,k)$-weak cotorsion (see Remarks \ref{1.33}). Moreover, we show that if $n$-${\rm wid}_{R}(R)\leq k$ and $M$ is an $R$-module, then
$M$ is an $(n,k)$-weak cotorsion  if and only if $M$ is a direct sum of an injective right $R$-module and a reduced $(n,k)$-weak cotorsion right $R$-module (see Theorem \ref{thm-M-nk-cotorsion}).

In Sec. 4, we prove some equivalent characterizations of rings with finite $n$-super finitely presented dimensions. Namely, for a coherent ring $R$ such that  $n$-${\rm wid}_{R}(_RR)\leq k$, we prove that  ${\rm l.n.sp. gldim}(R)\leq k$ if and only if 
every right $R$-module in $\mathcal{WI}_{ k-1}^{n}(R)^{\top}$ is $n$-weak flat if and only if every right $R$-module in $\mathcal{WI}_{ k}^{n}(R)^{\top}$ is $n$-weak flat if and only if every left $R$-module in $^{\bot}\mathcal{WI}_{ k}^{n}(R)$ belongs to $\mathcal{WI}_{ k}^{n}(R)$ if and only if every  left $R$-module in $\mathcal{WI}_{ k}^{n}(R)^{\bot}$ is $n$-weak injective if and only if every left $R$-module  has a monic $\mathcal{WI}_{ k-1}^{n}(R)$-cover (see Theorem \ref{Thm-lnspg-finite}). Then, we show that  ${\rm l.nsp. gldim}_{R}(R)\leq k$ if and only if every $(n, k)$-weak cotorsion right $R$-module is injective if and only if every $(n, k)$-weak cotorsion right $R$-module is in $\mathcal{WF}_{ k}^{n}(R^{op})$ (see Theorem \ref{Thm-lnspgldim-k}), and if every $(n, k)$-weak cotorsion right $R$-module has a $\mathcal{WF}_{ k}^{n}(R^{op})$-envelope with the unique mapping property, then ${\rm l.n.sp. gldim}(R)\leq k+2$ (see Proposition \ref{prop-lnspgldim-k+2}).
\section{Preliminaries}

In this section, some fundamental concepts are recalled and notations are stated.

A right $R$-module $M$ is said to be {\it cotorsion} \cite{EJ} if ${\rm Ext}_{R}^{1}(F,M)=0$ for any flat right $R$-module $F$.  For a non-negative integer $k$, a right $R$-module $M$ is said to be {\it $k$-cotorsion} \cite{LM} if ${\rm Ext}_{R}^{1}(F,M)=0$ for any right $R$-module $F\in\mathscr{F}_{k}(R^{op})$.  A left $R$-module $U$ is called {\it super finitely presented} \cite{Z.G} if there exists an exact sequence
$ \cdots\rightarrow  F_{2}\rightarrow  F_1\rightarrow  F_0\rightarrow  U\rightarrow  0$, where each $F_i$
is finitely generated and free.  A left $R$-module $M$ is called  {\it weak injective} \cite{Z.H} if ${\rm Ext}_{R}^{1}(U, M)=0$ for any super finitely presented left $R$-module $U$. A right $R$-module $M$ is called  {\it weak flat} \cite{Z.H} if ${\rm Tor}_{1}^{R}(M, U)=0$ for any super finitely presented left $R$-module $U$.  For a non-negative integer $k$, a right $R$-module $M$ is called  {\it $k$-weak cotorsion} \cite{B2B} if ${\rm Ext}_{R}^{1}(F, M)=0$ for any $F\in\mathcal{WF}_{ k}(R^{op})$. A left $R$-module $M$ is said to be {\it FP-injective} \cite{Su} if ${\rm Ext}_{R}^{1}(N,M)=0$ for any finitely presented left $R$-module $N$.

\begin{Def}[\cite{AB}, Definition 2.1]\label{Def-nsfp} 
Let  $n$ be a non-negative integer. A left $R$-module $U$ is said to be $n$-super finitely presented if there exists an exact sequence 
$$ \cdots\to F_{n+1}\to F_{n}\to\cdots\to F_1\to F_0\to U\to 0$$ of projective $R$-modules, where each $F_i$ is  finitely generated and projective for any $i\geq n$.
 If $K_{i}:={\rm Im}(F_{i+1}\to F_{i})$, then for $i=n-1$, the module $K_{n-1}$ is called special super finitely presented. 
\end{Def}
Notice that ${\rm Ext}_{R}^{n+1}(U,M)\cong{\rm Ext}_{R}^{1}(K_{n-1},M)$ and  ${\rm Tor}_{n+1}^{R}(N,U)\cong {\rm Tor}_{1}^{R}(N,K_{n-1})$, where $U$ is an $n$-super finitely presented left module with an associated special super finitely presented module $K_{n-1}$. This fact will be used throughout the paper.

$0$-Super finitely presented modules are just super finitely presented modules. Also for any $m\geq n$, every $n$-super finitely presented left $R$-module is $m$-super finitely presented but not conversely (see Example \ref{1.ya} and \cite[Examples 2.4 and 2.5]{AB}).
 
The finitely presented dimension of an $R$-module $A$ is defined as
${\rm f.p.dim}_{R}(A)={\inf}\{m \mid \text{there exists an exact sequence} \  {\rm F}_{n+1}\to {\rm F}_{n} \to \cdots \to {\rm F}_1\to {\rm F}_0\to {\rm A}\to 0$  of $R$-modules, where each ${\rm F}_{i}$ is projective, and ${\rm F}_{n}$ and ${\rm F}_{n+1}$ are finitely generated$\}$.
We also define the finitely presented dimension of $R$ (denoted by ${\rm f.p.dim}(R)$) as $\sup\{{\rm f.p.dim}_{R}(A) \ \mid \ A \ \text{is a finitely  generated} \ R\text{-module} \}$. 

Also, $R$ is called an $(a,b,c)$-ring if ${\rm w.gl.dim}(R)=a$, ${\rm gl.dim}(R)=b$ and ${\rm f.p.dim}(R)=c$ (see \cite{HKN}).
\begin{ex}\label{1.ya}
Let $R=R_{1}\oplus R_{2}$, where $R_{1}$ is a ring of polynomials in $4$ indeterminates over a field $k$, and $R_{2}$ is a
$(3,3,4)$-ring (see, \cite[Example 2.4]{AB} and \cite[Proposition 3.10]{HKN}).
 Then by \cite[Proposition 3.8]{HKN}, $R$  is a coherent $(4,4,4)$-ring. Then, ${\rm f.p.dim}(R)=4$ and so there exists a finitely generated
$R$-module $U$ with
   ${\rm f.p.dim}_{R}(U)=4$. Hence,  there exists an exact sequence
$F_{5}\rightarrow F_{4}\rightarrow F_{3}\rightarrow  F_{2}\rightarrow F_1\rightarrow F_0\rightarrow U\rightarrow 0,$
where $F_4$ and $F_{5}$ are finitely generated and projective $R$-modules. Also, $K_3:={\rm Im}( F_4\rightarrow F_3)$ is a special super finitely presented module, since $R$ is coherent. So by Definition \ref{Def-nsfp}, $U$ is $4$-super finitely presented. But, every $4$-super finitely presented is not $3$-super finitely presented otherwise   ${\rm f.p.dim}(R)=3$, a  contradiction. 
\end{ex}
\begin{Def}[\cite{AB}, Definition 2.2]\label{Def-nwi-nwf} 
Let $n$ be a non-negative integer. Then, a left $R$-module $M$ is called $n$-weak injective if ${\rm Ext}_{R}^{n+1}(U,M)=0$ for every $n$-super finitely presented left $R$-module $ U$. A right $R$-module $N$ is called $n$-weak flat
if ${\rm Tor}_{n+1}^{R}(N,U)=0$ for every $n$-super finitely presented left $R$-module $U$.
\end{Def}

$0$-Weak injective modules are just weak injective modules and  $0$-weak flat modules are just weak flat modules. And for any $m\geq n$, every  $n$-weak injective and  every $n$-weak flat module is $m$-weak injective and  $m$-weak flat respectively, but not
conversely (see Example \ref{1.3a}(1)) and \cite[Example  2.5(2)]{AB}.

 The left super finitely presented dimension ${\rm l.sp. gldim}(R)$ of $R$ is defined as
 ${\rm l.sp. gldim}(R):= {\rm sup}\{{\rm pd}_{R}(M) \mid M \ \text{is a super finitely presented left} \ R\text{-module}\}$, whereas the $n$-super finitely presented dimension of a ring $R$ is defined as follows:
${\rm l.n.sp. gldim}(R):={\rm sup}\{{\rm pd}_{R}(K_{n-1}) \mid K_{n-1} \ \text{is a special super finitely presented left} \ R\text{-module} \}$ (see \cite[Definition 4.9]{AB}).

It is easy to see that for any $n\geq 0$, ${\rm l.n.sp. gldim}(R)\leq{\rm l.sp. gldim}(R)$. Furthermore, Example \ref{1.3a}(3) shows that we do not have an equality, but if $n=0$, then clearly ${\rm l.n.sp. gldim}(R)={\rm l.sp. gldim}(R)$.

In \cite[Theorem 3.8]{Z.W}, it is proved that ${\rm l.sp. gldim}(R) \leq{\rm w.gl.dim}(R)$. Then, it is easy to obtain the following result.
 
\begin{prop}\label{prop-lnsp-wgldim}
 For any $n\geq 0$, ${\rm l.n.sp.gldim}(R)\leq{\rm w.gl.dim}(R)$.
\end{prop}


Let $\F$ be a class of $R$-modules and $M$ be an $R$-module. Following \cite{EM}, we say that a  morphism $f : F\to M$ is an
$\F$-precover of $M$ if $F\in\F$ and ${\Hom}_{R}(F^{'}, F) \to {\Hom}_{R}(F^{'},M)\to 0$ is exact
for any $F'\in\F$. An $\F$-precover $f:F\to M$ is said to be an $\F$-cover if every endomorphism $g : X \to X$ such
that $fg = f$ is an isomorphism. Dually, we have the definitions of an $\F$-preenvelope and an $\F$-envelope. The
class $\F$ is called (pre)covering if each $R$-module has an $\F$-(pre)cover. Similarly, if every
$R$-module has an $\F$-(pre)envelope, then we say that $\F$ is (pre)enveloping.

To any given class of right $R$-modules $\mathcal{L}$  and class of left $R$-modules $\mathcal{L'}$, we associate its orthogonal classes as follows:$${\mathcal{L}}^{\bot}={\Ker Ext}_{R}^{1}(\mathcal{L}, -)=\{C\in R\text{-}\Mod \mid {\Ext}_{R}^{1}(L, C)=0 \ \text{for any} \ L\in \mathcal{L}\},$$
$$^{\bot}{\mathcal{L}}={\Ker Ext}_{R}^{1}(-, \mathcal{L})=\{C\in R\text{-}\Mod \mid {\rm Ext}_{R}^{1}(C, L)=0 \ \text{for any} \ L\in\mathcal{L}\},$$
$${\mathcal{L}}^{\top}={\Ker \Tor}^{R}_{1}(\mathcal{L}, -)=\{C\in R\text{-}\Mod \mid {\rm Tor}^{R}_{1}(L, C)=0 \ \text{for any} \ L\in\mathcal{L}\},$$
$$^{\top}{\mathcal{L}'}={\Ker Tor}^{R}_{1}(-,\mathcal{L}')=\{C\in \Mod\text{-}R \mid {\rm Tor}^{R}_{1}(C, L)=0 \ \text{for any} \ L\in\mathcal{L'}\}.$$
A pair $(\F, \mathcal{C})$ of classes of  $R$-modules is  called a cotorsion theory \cite{EM} if $\F^{\bot}= \mathcal{C}$ and $\F= {^{\bot}\mathcal{C}}$.
A cotorsion theory $(\F, \mathcal{C})$ is called hereditary if whenever $0 \to F^{'}\to F\to F^{''}\to 0$ is exact with $F, F^{''}\in \F$ then $F^{'}\in\F$, or equivalently, if $0 \to C^{'}\to C\to C^{''}\to 0$ is an exact sequence with $C, C^{'}\in \mathcal{C}$, then $C^{''}\in\mathcal{C}$. A cotorsion theory $(\F, \mathcal{C})$ is called complete \cite{TAT} if every $R$-module has a special $\mathcal{C}$-preenvelope; that is, a monic $\mathcal{C}$-preenvelope with cokernel in $^{\bot}\mathcal{C}$ (and a special $\F$-precover; that is, an epic $\F$-precover with kernel in $\F^{\bot}$). A cotorsion theory $(\F, \mathcal{C})$ is called perfect  if every $R$-module has a $\mathcal{C}$-envelope and an $\F$-cover (see \cite{EOO,ZGR}).
 
 \begin{Def}[\cite{AB}, Definition 3.1]\label{2.1}
 	Let $n$ be a non-negative integer. Then, the $n$-weak injective dimension of a left module $M$ is defined by:
 	
 	$n$-${\rm wid}_{R}(M):= {\inf}\{k: \ { \Ext}_{R}^{k+1}(K_{n-1}, M)=0 \ \text{for every special super finitely presented left module} \ K_{n-1}\}$, and the n-weak flat dimension of a right module N is defined by: 
 	
$n$-${\rm wfd}_{R}(N):= \inf\{k: \ {\rm Tor}^{R}_{k+1}(N, K_{n-1})=0 \ \text{for every special super finitely presented left module} \ K_{n-1}\}$.
 \end{Def}
 
 The classes of left and right $R$-modules with $n$-weak injective and $n$-weak flat dimensions at most $k$ will be denoted respectively by $\mathcal{WI}_{ k}^n(R)$ and $\mathcal{WF}_{ k}^n(R^{op})$. In \cite[Theorem 4.4 and 4.5]{AB}, it is proved that $\mathcal{WI}_{ k}^n(R)$ and $\mathcal{WF}_{ k}^n(R^{op})$ are both covering and preenveloping. Here, we give more properties of these two classes of modules. We start with the following result.

\begin{prop}\label{1.5q}
Let $n$ and $k$ be non-negative integers. Then, the following assertions hold:
\begin{enumerate}
\item [\rm (1)] 
If $M$ is the cokernel of a $\mathcal{WI}_{k}^n(R)$-preenvelope $X\to F$ of a left  R-module $X$ with $F$  flat, then $M\in {}^{\top}\mathcal{WF}_{ k}^n(R^{op})$.
\item [\rm (2)] 
If $M$ is the cokernel of a $\mathcal{WF}_{ k}^n(R^{op})$-preenvelope $A\to F$ of a right  R-module $A$ with $F$  flat, then $M \in \mathcal{WI}_{ k}^n(R)^{\top}$.
\end{enumerate}
\end{prop}
\begin{proof}
(1) Consider the short exact sequence $0\to X\to F\to M\to 0$. If $N\in\mathcal{WF}_{ k}^n(R^{op})$, then by \cite[Proposition 3.6]{AB}, $N^*\in\mathcal{WI}_{k}^n(R)$. Thus, ${\Hom}_{R}(F, N^*)\to {\Hom}_{R}(X, N^*)\to 0$ is exact and so
 $ (F\otimes_{R}N)^{*}\to (X\otimes_{R}N)^{*}\to 0$ is exact. Hence, $ 0\to X\otimes_{R}N \to  F\otimes_{R}N $ is exact and from the exactness of the long exact sequence $ 0\to{\rm Tor}_{1}^{R}(M,N)\to X\otimes_{R}N \to  F\otimes_{R}N $ (because $F$ is flat), it follows that ${\rm Tor}_{1}^{R}(M,N)=0$  and so $M\in {}^{\top}\mathcal{WF}_{ k}^n(R^{op})$.

 (2) Consider the exact sequence $0\to X\to F\to M\to 0$, where $X=\Im(A\to F)$. By using similar arguments in (1), we prove that $M\in\mathcal{WI}_{ k}^n(R)^{\top}$.
\end{proof} 
\begin{cor}\label{1.5qq}
Let $n$ and $k$ be non-negative integers. Then, $^{\bot}\mathcal{WF}_{ k}^n(R^{op})\subseteq\mathcal{WI}_{ k}^n(R)^{\top}$.
\end{cor}
\begin{proof}
Consider an exact sequence $0\to X\to F\to N\to 0$, where $N\in {}^{\bot}\mathcal{WF}_{ k}^n(R^{op})$ and $F$ is
projective. Then, ${\Ext}_{R}^{1}(N, M)=0$ for any $M\in\mathcal{WF}_{ k}^n(R^{op})$. Hence, $ X\to F$ is a $\mathcal{WF}_{ k}^n(R^{op})$-preenvelope of $X$, and so $N\in\mathcal{WI}_{ k}^n(R)^{\top}$ by Proposition \ref{1.5q}(2).
\end{proof}

\begin{thm}\label{1.5qqq}
Let $n$ and $k$ be non-negative integers and $U$ be an $n$-super finitely presented left module. If $\cdots\to F_{n+1}\to F_{n}\to\cdots\to F_1\to F_0\to U\to 0$ is a projective resolution such that for any $i\geq n$, $F_i$ is finitely generated and projective and $K_{i}={\rm Im}(F_{i+1}\to F_{i})$,
then the following assertions hold:
\begin{enumerate}
\item [\rm (1)] 
If ${\rm n\text{-}wid}(_RR)\leq k$, then  $K_{n-1}\in{{}^{\top}}\mathcal{WF}_{ k}^n(R^{op})$ if and only if  $K_{n}\to F_n$   is a $\mathcal{WI}_{k}^n(R)$-preenvelope.
\item [\rm (2)] $K_{n-1}\in\mathcal{WI}_{ k}^n(R)^{\top}$  if and only if $K_{n}\to F_n$ is a $\mathcal{WF}_{k}^n(R^{op})$-preenvelope if and
only if $K_{n-1}\in{^{\bot}}\mathcal{WF}_{ k}^n(R^{op})$.
\end{enumerate}
\end{thm}
\begin{proof}
(1) Consider the exact sequence
$0\to K_{n}\to F_{n}\to K_{n-1}\to 0$ with $F_n$ is finitely generated and projective. It is clear by \cite[Proposition 4.8]{AB} that $F_{n}\in\mathcal{WI}_{ k}^n(R)$ (because ${\rm n\text{-}wid}(_RR)\leq k$). We show that $K_{n}\to F_n$ is a $\mathcal{WI}_{k}^n(R)$-preenvelope. For any $X\in\mathcal{WI}_{ k}^n(R)$, we have $X^{*}\in\mathcal{WF}_{ k}^n(R^{op})$ by \cite[Proposition 3.6(2)]{AB}. So $ {\rm Tor}^{R}_{1}(K_{n-1}, X^{*})=0$. By \cite[Lemma 3.55]{Rot1}, we obtain the following commutative diagram
$$\xymatrix{
 0\ar[r]&  K_{n}\otimes_{R}X^{*}\ar[r]^{\alpha}\ar[d]^{f_2}&  F_{n}\otimes_{R}X^{*}\ar[d]^{f_1}&  \\
  &{\rm Hom}_{R}(K_{n}, X)^{*}\ar[r]^{\beta}&{\rm Hom}_{R}(F_{n}, X)^{*}.&  \\
}$$
where $f_1$ and $f_2$ are isomorphisms. We deduce that $\beta$ is a monomorphism and consequently the sequence ${\rm Hom}_{R}(F_{n}, X)\to{\rm Hom}_{R}(K_{n}, X)\to0$ is exact. The converse follows from Proposition \ref{1.5q}(1).

 (2) The proof of the first equivalence is similar to that of (1) using Proposition  \ref{1.5q}(2). The second equivalence is clear using Corollary  \ref{1.5qq}(2).
\end{proof}

\section{$(n,k)$-Weak cotorsion modules}
\ \ \  In this section, we introduce and study $(n,k)$-weak cotorsion modules. We start with the definition.

\begin{Def}\label{Def-nk-weak-cot} 
Let  $n$ and $k$ be  non-negative integers. Then, a right $R$-module $M$ is called $(n,k)$-weak cotorsion if ${\rm Ext}_{R}^{1}(N,M)=0$ for every $N\in\mathcal{WF}_{ k}^n(R^{op})$. 
\end{Def}

\begin{rem}\label{1.33}
\begin{enumerate}
\item [\rm (1)]
$(0,k)$-Weak cotorsion right modules are exactly $k$-weak cotorsion modules, and $(0,0)$-weak cotorsion right modules are exactly weak cotorsion modules.
\item [\rm (2)]
Every $(n,k)$-weak cotorsion right module is $(m,k)$-weak cotorsion for any $n\geq m$, but not conversely. Indeed, if  every $(m,k)$-weak cotorsion right module $M$ is $(n,k)$-weak cotorsion, then either $M$ is injective or for each $N\in \mathcal{WF}_{ k}^n(R^{op})$,  $N$ is projective or $N\in\mathcal{WF}_{ k}^{n-1}(R^{op})$, and this is impossible (see Example \ref{1.3a}(2)).
\item [\rm (3)] The study of the class of $(n,k)$-weak cotorsion right modules involve the classes of  injective,  $(n,k)$-weak cotorsion, $k$-weak cotorsion, $k$-cotorsion and cotorsion right $R$-modules denoted $\mathcal{I}$,  $\mathcal{WC}_k^n(R^{op})$, $\mathcal{WC}_k(R^{op})$, $\mathcal{C}_k(R^{op})$ and $\mathcal{C}(R^{op})$ respectively. In fact, $$\mathcal{I}\subseteq\mathcal{WC}_k^n(R^{op})\subseteq\mathcal{WC}_k(R^{op})\subseteq\mathcal{C}_k(R^{op})\subseteq\mathcal{C}(R^{op}).$$
But not every $k$-cotorsion right module is $(n,k)$-weak cotorsion (see Example \ref{1.3a}(3)).
\end{enumerate}
\end{rem}

\begin{ex}\label{1.3a}
Let $R=k[x]\times S$, where $k[x]$ is a ring of polynomials in one indeterminate over a field $k$ and $S$ is a non-Noetherian hereditary von Neumann regular ring. For example, $S$ is a ring of functions of $X$ into $k$ continuous with respect to the discrete topology on $k$, where $k$ is a field and $X$ is a totally disconnected compact Hausdorff space whose associated Boolean ring is hereditary (see examples of \cite{GMB}).
 Then, by \cite[Proposition 3.10]{HKN}, $R$  is a coherent $(1,1,2)$-ring. Then, we have:
\begin{enumerate}
\item [\rm (1)]
Every $R$-module is $1$-weak flat, since ${\rm gl.dim}(R)=1$. But not every $R$-module is $0$-weak flat, for otherwise each super finitely presented $R$-module would be projective. On the other hand, every finitely presented is super finitely presented (because $R$ is coherent). Thus every finitely presented is flat and by \cite[Theorem 3.9]{.TZ}, each $R$-module is flat and so  ${\rm w.gl.dim}(R)=0$,   a  contradiction. 
\item [\rm (2)]
 Notice that ${\rm l.n.sp.gldim}(R)=0$ for any $n\geq 1$, since ${\rm pd}_{R}(U)\leq 1$ for any $n$-super finitely presented $R$-module $U$. So by Theorem \ref{Thm-lnspgldim-k},  every $(n,0)$-weak cotorsion right $R$-module is  injective  for any $n\geq 1$. But not every $(0,0)$-weak cotorsion right $R$-module is  injective, for otherwise every $R$-module would be in $\mathcal{WF}_0^0(R^{op})$ and so each $0$-super finitely presented would be flat. Similarly to (1), ${\rm w.gl.dim}(R)=0$,   a  contradiction.
\item [\rm (3)]
 Also, ${\rm l.sp.gldim}(R)=1$, since $R$ is coherent and by \cite[Theorem 3.8]{Z.W},  ${\rm l.sp.gldim}(R)={\rm w.gl.dim}_{R}(R)=1$. So ${\rm l.n.sp.gldim}(R)\neq {\rm l.sp.gldim}(R)$ and it follows by Proposition \ref{1.1001} that not every  $k$-cotorsion  is $(n,k)$-weak cotorsion.
\end{enumerate}
\end{ex}
Example \ref{1.3a}(3) gives a context when ${\rm l.nsp. gldim}(R)\leq {\rm l.sp. gldim}(R)$. 

Now we show, for non-negative integers $k$ and $n$, when every $k$-cotorsion module is $(n,k)$-weak cotorsion.
\begin{prop}\label{1.1001}
Let $n$ be a non-negative integer. Then, the following conditions are equivalent: 
\begin{enumerate}
\item [\rm (1)] 
 For every non-negative integer $k$, every $k$-cotorsion right module is $(n,k)$-weak cotorsion.
\item [\rm (2)] 
For every non-negative integer $k$, $\mathcal{WF}_{ k}^n(R^{op})=\mathscr{F}_{k}(R^{op})$.
\item [\rm (3)]
${\rm l.n.sp. gldim}(R)={\rm w. gl.dim}(R)$.
\end{enumerate}
\end{prop}
\begin{proof}
The implication $(1)\Rightarrow (2)$ follows from the fact that $(\mathcal{WF}_{ k}^n(R^{op})), \mathcal{WC}_{ k}^n(R^{op}))$  and $(\mathscr{F}_{ k}(R^{op}), \mathcal{C}_{ k}(R^{op}))$ are cotorsion theories. And, the implications  
$(2)\Leftrightarrow (3)\Rightarrow(1)$ are clear.
\end{proof}
The following result can be easily obtained.

\begin{prop}\label{1.5}
	
	Let $n$ and $k$ be non-negative integers. Then, the following assertions hold:
\begin{enumerate}
\item [\rm (1)] 
If $\{M_{i}\}_{i\in I}$ is a family of right $R$-modules, then
$\prod_{i\in I} M_{i}$ is $(n,k)$-weak cotorsion if and only if each $M_i$ is  $(n,k)$-weak cotorsion.
\item [\rm (2)]  
The class $\mathcal{WC}_k^n(R^{op})$ is closed under extensions and direct summands.
\end{enumerate}
\end{prop}

Now we state one of the main results of this section which characterizes when a module $M$ is $(n,k)$-weak cotorsion. We start with the following proposition.

\begin{prop}\label{1.666}
Let $n$ and $k$ be non-negative integers and let $M$ be a right $R$-module. Then, the following conditions are equivalent:
\begin{enumerate}
\item [\rm (1)] 
$M$ is $(n,k)$-weak cotorsion.
\item [\rm (2)] 
$M$ is injective with respect to every  exact sequence $0\to K\to N\to D\to 0$ of right $R$-modules, where $D\in\mathcal{WF}_{k}^n(R^{op})$.
\end{enumerate}

Moreover if ${\rm n\text{-}wid}(_RR)\leq k$, then the above conditions are also equivalent to:
\begin{enumerate}
\item [\rm (3)]
For every exact sequence $0\to M\to E\to D\to 0$ of right $R$-modules, where $E$ is injective, $E\to D$ is a $\mathcal{WF}_{k}^n(R^{op})$-precover of $D$.
\item [\rm (4)]
$M$ is the kernel of a $\mathcal{WF}_{k}^n(R^{op})$-precover $E\stackrel{\displaystyle \phi}\to D$ where $E$ is injective.
\end{enumerate}
\end{prop}
\begin{proof}
$(1)\Rightarrow (2)$ Clear.

$(2)\Rightarrow (1)$
Let $D\in\mathcal{WF}_{k}^n(R^{op})$ and consider an exact sequence $0\rightarrow K\rightarrow P\rightarrow D\rightarrow 0$ where $P$ projective. Then, from the long exact sequence $ {\Hom}_{R}(P,M)\rightarrow  {\Hom}_{R}(K,M)\rightarrow   {\rm Ext}_{R}^1(D,M)\rightarrow   {\rm Ext}_{R}^1(P,M)$ we see that ${\rm Ext}_{R}^{1}(D, M)=0$ and hence $M$ is 
$(n,k)$-weak cotorsion.

$(1)\Rightarrow (3)$
For every exact sequence $0\rightarrow M\rightarrow E\stackrel{\displaystyle \phi}\rightarrow D\rightarrow 0$ where $E$ is injective, we have by \cite[Proposition 4.8]{AB} that $E\in\mathcal{WF}_{k}^n(R^{op})$.
So for any $N$ in $\mathcal{WF}_{k}^n(R^{op})$, we have
$$0\rightarrow {\Hom}_{R}(N,M)\rightarrow  {\Hom}_{R}(N,E)\rightarrow  {\Hom}_{R}(N,D)\rightarrow  {\rm Ext}_{R}^1(N,M)=0.$$
It follows that $ \phi$ is a $\mathcal{WF}_{k}^n(R^{op})$-precover of $D$.

$(3)\Rightarrow (4)$ Trivial.

$(4)\Rightarrow (1)$
By hypothesis, there exists an exact sequence $0\rightarrow M\rightarrow E\stackrel{\displaystyle \phi}\rightarrow D\rightarrow 0$, where $\phi$ is a $\mathcal{WF}_{k}^n(R^{op})$-precover of $D$. Therefore for each $L$ in $\mathcal{WF}_{k}^n(R^{op})$, we have:
$$0\rightarrow {\Hom}_{R}(L,M)\rightarrow  {\Hom}_{R}(L,E)\rightarrow  {\Hom}_{R}(L,D)\rightarrow  {\rm Ext}_{R}^1(L,M)\rightarrow  {\rm Ext}_{R}^1(L,E),$$
We have ${\Hom}_{R}(L,E)\rightarrow  {\Hom}_{R}(L,D)$ is epic and ${\rm Ext}_{R}^1(L,E)=0$  so ${\rm Ext}_{R}^1(L,M)=0$ and consequently $M$ is $(n,k)$-weak cotorsion.
\end{proof}
Recall that an $R$-module $N$ is said to be reduced \cite{EM} if $N$ has no non zero injective submodules. In the next proposition, we characterize reduced $(n,k)$-weak cotorsion modules when ${\rm n\text{-}wid}(_RR)\leq k$.
\begin{prop}\label{1.76}
Let $n$ and $k$ be non-negative integers and assume that ${\rm n\text{-}wid}(_RR)\leq k$. Then for a right $R$-module $M$, 
$M$ is a reduced $(n,k)$-weak cotorsion if and only if $M$ is the kernel of a $\mathcal{WF}_{k}^n(R^{op})$-cover $N\stackrel{\displaystyle \phi}\rightarrow D$ where $N$ is an injective module.
\end{prop}
\begin{proof}
Suppose that $M$ is a reduced $(n,k)$-weak cotorsion right module and consider the exact sequence $0\rightarrow M\rightarrow E(M)\stackrel{\displaystyle \phi}\rightarrow \frac{E(M)}{M}\rightarrow 0$. Then, by Proposition \ref{1.666} $\phi$ is a $\mathcal{WF}_{k}^n(R^{op})$-precover. Since $M$ is reduced, $E(M)$ has no non zero direct summand $L$ contained in $M$. Also by \cite[Theorem 4.5]{AB}, $\frac{E(M)}{M}$ has a  $\mathcal{WF}_{k}^n(R^{op})$-cover. Then, \cite[Corollary 1.2.8]{JX} implies that $\phi$  is a $\mathcal{WF}_{k}^n(R^{op})$-cover.

Conversely, suppose that $M$ is the kernel of a $\mathcal{WF}_{k}^n(R^{op})$-cover $N\stackrel{\displaystyle \phi}\rightarrow D$ where $N$ is an injective $R$-module. Then by Proposition \ref{1.666}, $M$ is $(n,k)$-weak cotorsion. Let $N_1$ be an injective submodule of $M$. If $N=N_1\oplus N_2$, $f: N\rightarrow N_2$ is  the projection map and $i: N_2\rightarrow N $ is the inclusion map, then $\phi(N_1)=0$ and  $\phi(if)=\phi$, and hence $if$ is an isomorphism.  So $i$ is an epimorphism and consequently $N=N_2$  and $N_1=0$. Therefore, $M$ is reduced.
\end{proof}Now we are in a position to show the following theorem.
\begin{thm}\label{thm-M-nk-cotorsion}
Let $n$ and $k$ be non-negative integers and assume that ${\rm n\text{-}wid}(_RR)\leq k$. Then for a right $R$-module $M$, $M$ is $(n,k)$-weak cotorsion  if and only if $M$ is a direct sum of an injective right $R$-module and a reduced $(n,k)$-weak cotorsion right $R$-module.
\end{thm}
\begin{proof}
Consider the exact sequence $0\rightarrow M\rightarrow E(M)\stackrel{\displaystyle \phi}\rightarrow \frac{E(M)}{M}\rightarrow 0$, where $M$ is an $(n,k)$-weak cotorsion right module. By Proposition \ref{1.76},  $\phi$ is a
 $\mathcal{WF}_{k}^n(R^{op})$-cover of $\frac{E(M)}{M}$. Also by \cite[Theorem 4.5]{AB}, $\frac{E(M)}{M}$ has a 
 $\mathcal{WF}_{k}^n(R^{op})$-cover $N\rightarrow \frac{E(M)}{M}$. Then, there exists the commutative diagram with exact
rows:
$$\xymatrix{
	0\ar[r]&N_1\ar[r]^{g_1}\ar[d]^{h_1}&N\ar[r]\ar[d]^{h_2}&\frac{E(M)}{M} \ar[r]\ar@{=}[d]&0 \\
	0\ar[r]&M\ar[r]^{g_2}\ar[d]^{f_1}& E(M)\ar[r]\ar[d]^{f_2}&\frac{E(M)}{M}\ar[r]\ar@{=}[d]& 0 \\
	0\ar[r]&N_1\ar[r]^{g_3}&N\ar[r]&\frac{E(M)}{M} \ar[r]&0
}$$
We have $E(M)={\rm ker}f_{2}\oplus {\rm im} h_2$, since $f_{2}h_{2}$ is an isomorphism. Also ${\rm im}h_{2}\cong N$ so $N$ and ${\rm ker}f_{2}$ are injective. Thus by Proposition \ref{1.76}, $N_1$ is reduced and $(n,k)$-weak cotorsion. On the other hand, by the Five Lemma, $f_1h_1$ is an isomorphism. Hence, $M={\rm ker}f_{1}\oplus {\rm im} h_1$ where ${\rm im} h_1\cong N_1$. Consider the following commutative diagram:
$$\xymatrix{
&0\ar[d]&0\ar[d]&0\ar[d]&\\
0\ar[r]&{\rm ker}f_{1}\ar[r]\ar[d]&{\rm ker}f_{2}\ar[r]\ar[d]&0 \ar[r]\ar[d]&0  \\
0\ar[r]&M\ar[r]^{g_2}\ar[d]^{f_1}& E(M)\ar[r]\ar[d]^{f_2}&\frac{E(M)}{M}\ar[r]\ar@{=}[d]& 0 \\
0\ar[r]&N_1\ar[r]^{g_3}\ar[d]&N\ar[r]\ar[d]&\frac{E(M)}{M} \ar[r]\ar[d]&0\\
&0&0&0&\\}$$
Consequently ${\rm ker}f_{1}\cong{\rm ker}f_{2}$ and so $M=N_1\oplus{\rm ker}f_{1}$, where  ${\rm ker}f_{1}$ is injective.
The converse follows from Proposition \ref{1.5}(1).
\end{proof}

We end this section with a result that shows when every right $R$-module is $(n,k)$-weak cotorsion. To this end, we prove first the following proposition.

\begin{prop}\label{1.66}
	The following assertions hold:
	\begin{enumerate}
		\item [\rm (1)] 
		If a right $R$-module $M$  is $(n,k)$-weak cotorsion, then ${\rm Ext}_{R}^{j\geq m+1}(N, M)=0$ for every $N\in\mathcal{WF}_{m+k}^n(R^{op})$.
		\item [\rm (2)] 
		The $m$th cosyzygy of every $(n,k)$-weak cotorsion right $R$-module is $(n, m+k)$-weak cotorsion.
	\end{enumerate}
\end{prop}

\begin{proof}
	(1) Let $N\in\mathcal{WF}_{m+k}^n(R^{op})$. Then, there is an exact sequence:
	$$ 0\to K_{m}\to P_{m-1}\to\cdots\to P_1\to P_0\to N\to 0,$$
	where each $P_i$ is projective and $K_m\in\mathcal{WF}_{k}^n(R^{op})$. Hence, 
	${\rm Ext}_{R}^{m+1}(N, M)\cong{\rm Ext}_{R}^{1}(K_{m}, M)=0$ (because $M$  is $(n,k)$-weak cotorsion).
	
	(2) Let $M$ be any $(n,k)$-weak cotorsion right $R$-module and $D^m$
	the $m$th cosyzygy of $M$. Then by (1), ${\rm Ext}_{R}^{m+1}(N, M)=0$ for every  $N\in\mathcal{WF}_{m+k}^n(R^{op})$. Thus, $0={\rm Ext}_{R}^{m+1}(N, M)\cong{\rm Ext}_{R}^{1}(N, D^m)$ and consequently $D^m$ is $(n,m+k)$-weak cotorsion.
\end{proof}

\begin{thm}\label{1.100}
Let $n$ and $k$ be non-negative integers. Then, the following conditions are equivalent:
\begin{enumerate}
\item [\rm (1)] 
Every right $R$-module is $(n,k)$-weak cotorsion.
\item [\rm (2)] 
Every right $R$-module in $\mathcal{WF}_{k}^n(R^{op})$ is projective.
\item [\rm (3)]
Every right $R$-module in $\mathcal{WF}_{k}^n(R^{op})$ is $(n,k)$-weak cotorsion.
\item [\rm (4)]
For any integer $m$, ${\rm Ext}^{i\geq m+1 }_R(M,N)=0$ for all  right $R$-modules $N$ and all $R$-module $M$ in $\mathcal{WF}_{m+k}^n(R^{op})$.
\item [\rm (5)]
Every right $R$-module $M$ has a $\mathcal{WF}_{k}^n(R^{op})^{\bot}$-envelope with the unique mapping property.
\item [\rm (6)]
Every projective right $R$-module is $(n,k)$-weak cotorsion.
\item [\rm (7)]
$R_R$ is $(n,k)$-weak cotorsion and every right $R$-module has a  $\mathcal{WF}_{k}^n(R^{op})^{\bot}$-precover.
\end{enumerate}
\end{thm}

\begin{proof}
$(1)\Leftrightarrow (2)$ Clear because $(\mathcal{WF}_{ k}^n(R^{op})), \mathcal{WC}_{ k}^n(R^{op}))$  and $(Proj, \Mod\text{-}R)$ are cotorsion theories.

$(1)\Rightarrow (4)$ Follows from Proposition \ref{1.66}(1).

$(4)\Rightarrow (3)$  Take $m=0$.

$(3)\Rightarrow (1)$
Let $M$ be a right $R$-module. Then by \cite[Proposition  4.13]{AB}, there exists an exact sequence 
 $0\rightarrow L\rightarrow N\stackrel{\displaystyle \phi}\rightarrow M\rightarrow 0$, where $\phi$ is $\mathcal{WF}_{k}^n(R^{op})$-cover of $M$. Also by \cite[Lemma 2.1.1]{JX}, $L$ is in 
$\mathcal{WF}_{k}^n(R^{op})^{\bot}=\mathcal{WC}_{k}^n(R^{op})$. Since $(\mathcal{WF}_{k}^n(R^{op}), \mathcal{WF}_{k}^n(R^{op})^{\bot})$ is a hereditary perfect cotorsion pair, we deduce that every right $R$-module $M$ is $(n,k)$-weak cotorsion.

$(1)\Rightarrow (5)$ Clear.

$(5)\Rightarrow (3)$ 
Let $N\in\mathcal{WF}_{k}^n(R^{op})$. Then,  the following commutative diagram exists:
\begin{displaymath}\xymatrix{
&&&0\ar[d]\\
0\ar[r]&N \ar[r]^{\alpha} \ar[rd]_0& X \ar[r]^{\beta}\ar[d]^{\gamma\beta}& D\ar[r]\ar[ld]^{\gamma}& 0, \\
 &&X^{'}}
\end{displaymath}
where $\alpha$ and $\gamma$  are $\mathcal{WF}_{k}^n(R^{op})^{\bot}$-envelope. Since $\gamma\beta\alpha=0=0\alpha$, it follows by assumption that $\gamma\beta=0$. Consequently 
$D={\rm im}(\beta)\subseteq {\rm ker}(\gamma)=0$ and so $D=0$. Hence, $N$ is $(n,k)$-weak cotorsion.

$(1)\Rightarrow (7)$  Clear.

 $(7)\Rightarrow (6)$
By \cite[Proposition 1]{Z.JJG}, $\mathcal{WF}_{k}^n(R^{op})^{\bot}$ is closed under direct sums. Since $R$ is $(n,k)$-weak cotorsion, it follows that every free right $R$-module is $(n,k)$-weak cotorsion. Then by Proposition \ref{1.5}(2), we deduce that each projective right $R$-module is $(n,k)$-weak cotorsion.

$(6)\Rightarrow (1)$
Let $M$ be a right $R$-module. Then by \cite[Proposition 4.13]{AB}, there exists an exact sequence 
 $0\rightarrow L\rightarrow N\stackrel{\displaystyle \alpha}\rightarrow M\rightarrow 0$, where $\alpha$ is $\mathcal{WF}_{k}^n(R^{op})$-cover of $M$ and $L\in\mathcal{WF}_{k}^n(R^{op})^{\bot}$ by \cite[Lemma 2.1.1]{JX}.  By hypothesis, every projective right $R$-module is $(n,k)$-weak cotorsion, and then by Remark \ref{1.33}[(1), (2)], every projective right $R$-module is  cotorsion. Hence by \cite[Corollary 10]{GAA}, $R$ is perfect. So by \cite[Theorem Bass]{Rot2}, any flat right $R$-module $Y$ is projective and from the assumption, $Y$ is $(n,k)$-weak cotorsion.  Since every flat module is $n$-weak flat and $N$ is in  $\mathcal{WF}_{k}^n(R^{op})$, we deduce that for any flat resolution of $N$,  there is an exact sequence $0\rightarrow F_{k}\rightarrow F_{k-1}\rightarrow\cdots \rightarrow F_{1}\rightarrow F_{0}\rightarrow N\rightarrow0,$ where each $F_i$ is  $(n,k)$-weak cotorsion. If $X_i={\rm im}(F_{i}\rightarrow F_{i-1})$ for any $1\leq i\leq k$, then by \cite[Proposition 4.13]{AB}, every $X_i$ is  $(n,k)$-weak cotorsion, and so the  exact sequence $0\rightarrow X_{1}\rightarrow F_{0}\rightarrow N\rightarrow0$ implies that $N$ is $(n,k)$-weak cotorsion (because $X_1$ and $F_0$ are $(n,k)$-weak cotorsion). Similarly, $M$ is $(n,k)$-weak cotorsion with respect to the exact sequence  $0\rightarrow L\rightarrow N\stackrel{\displaystyle \alpha}\rightarrow M\rightarrow 0$.
\end{proof}
\section{Applications}
\ \ \ In this section, we characterize rings with finite $n$-super finitely presented dimension in terms of $n$-weak injective and $n$-weak flat modules and $(n,k)$-weak cotorsion modules.\\
First, we start with the following lemmas.
\begin{lem}\label{1.105}
Let $n$ and $k$ be non-negative integers. Then, the following conditions are equivalent for a left $R$-module $N$:
\begin{enumerate}
\item [\rm (1)] 
$N\in\mathcal{WI}_{ k}^{n}(R)^{\top}$.
\item [\rm (2)] 
$N^*\in\mathcal{WI}_{ k}^{n}(R)^{\bot}$.
\item [\rm (3)] 
$N\in^{\bot}\X,$ where $\X=\{A^{*} \ \mid \ A \ is\ in\ \mathcal{WI}_{ k}^{n}(R)\}$.
\item [\rm (4)] 
The functor $-\otimes_{R}N$ preserves the exactness of every exact sequence $0\rightarrow A\rightarrow B\rightarrow C\rightarrow 0$ where $C\in\mathcal{WI}_{ k}^{n}(R)$.

\end{enumerate}
\end{lem}
\begin{proof}
By \cite[Theorem 11.55]{Rot1}, for any right $R$-module $D$, ${\rm Ext}_{R}^1(N, D^*)\cong{\rm Tor}_{1}^R(D, N)^*\cong{\rm Ext}_{R}^1(D, N^*)$. So  $(1)\Leftrightarrow (2) \Leftrightarrow (3)$ follows.
Also, $(1)\Leftrightarrow (4)$ is clear.
\end{proof}
\begin{lem}\label{1.104}
Assume that ${\rm n\text{-}wid}(_RR)\leq k$ where $k\geq 1$. Then,
\begin{enumerate}
\item [\rm (1)] 
If $N\in\mathcal{WI}_{ k-1}^{n}(R)^{\bot}$, then there is an exact sequence 
 $0\rightarrow X\rightarrow E\rightarrow N\rightarrow 0$ where $ E$ is injective and $X\in\mathcal{WI}_{k}^{n}(R)^{\bot}$.
\item [\rm (2)] 
If  $N\in\mathcal{WI}_{ k-1}^{n}(R)^{\top}$, then there is an exact sequence 
 $0\rightarrow N\rightarrow F\rightarrow D\rightarrow 0$ where $ F$ is flat and $D\in\mathcal{WI}_{ k}^{n}(R)^{\top}$.
\end{enumerate}
\end{lem}
\begin{proof}
(1)
 Consider an exact sequence $0\rightarrow Y\rightarrow F\rightarrow N\rightarrow0$ of left $R$-modules, where $F$ is projective. Then, we have the following pushout diagram
$$\xymatrix{
&&0\ar[d]&0\ar[d]&\\
0\ar[r]&Y\ar[r]\ar@{=}[d]&F\ar[r]\ar[d]^{\alpha}&N \ar[r]\ar[d]&0  \\
0\ar[r]&Y\ar[r]& E(F)\ar[r]\ar[d]&L\ar[r]\ar[d]& 0 \\
&&B\ar@{=}[r]\ar[d]&B\ar[d]&\\
&&0&0&\\
}$$
where $\alpha$ is an injective envelope of $F$.  Let $U$ be an $n$-super finitely presented left $R$-module with associated special super finitely presented module $K_{n-1}$. Then by assumption and from \cite[Proposition 3.2]{AB}, ${\rm Ext}_{R}^{k+1}(K_{n-1}, R)=0$ so ${\rm Ext}_{R}^{k+1}(K_{n-1}, F)=0$ and hence ${\rm n\text{-}wid}_{R}(F)\leq k$. Also by \cite[Proposition 3.2]{AB}, it follows that ${\rm n\text{-}wid}_{R}(B)\leq k-1$, and so 
$B\in\mathcal{WI}_{ k-1}^{n}(R)$. Hence, ${\rm Ext}_{R}^{1}(B, N)=0$, and consequently
 the sequence $0\rightarrow N\rightarrow L\rightarrow B\rightarrow0$ splits and then $N$ is a quotient of $E(F)$.

By \cite[Theorem 3.1]{Z.H}, there is a   weak injective cover of  left $R$-module $N$.  If  $\gamma: E\rightarrow N$ is a weak injective cover of $N$, then $\gamma$ is epic. Now consider the  exact sequence $0\rightarrow X\rightarrow E\rightarrow N\rightarrow 0$. By \cite[Lemma 2.1.1]{JX}, $X\in\mathcal{WI}_{0}^{0}(R)^{\bot}$. We prove that $X\in\mathcal{WI}_{k}^{n}(R)^{\bot}$. For let $L_1$ be in $\mathcal{WI}_{ k}^{n}(R)$ and consider the exact sequence $0\rightarrow L_1\rightarrow E(L_1)\rightarrow L_2\rightarrow 0$, where 
$L_2\in\mathcal{WI}_{ k-1}^{n}(R)$ by \cite[Proposition 3.2]{AB}. Then we get the induced exact sequence 
$${\rm Exr}^{1}_{R}(L_2, N)\to {\rm Exr}^{2}_{R}(L_2, X)\to{\rm Exr}^{2}_{R}(L_2, E)=0.$$
By hypothesis, we have ${\rm Exr}^{1}_{R}(L_2, N)=0$  and so ${\rm Exr}^{2}_{R}(L_2, X)=0$. Also,  we obtain the induced exact sequence
$${\rm Exr}^{1}_{R}(E(L_1), X)\to {\rm Exr}^{1}_{R}(L_1, X)\to{\rm Exr}^{2}_{R}(L_2, X).$$
We have ${\rm Exr}^{2}_{R}(L_2, X)=0$ and since $X$  is in $\mathcal{WI}_{0}^{0}(R)^{\bot}$ and $E(L_1)\in\mathcal{WI}_{0}^{0}(R)$,  it follows that ${\rm Exr}^{1}_{R}(E(L_1), X)=0$. Hence, ${\rm Exr}^{1}_{R}(L_1, X)=0$ and consequently $X\in\mathcal{WI}_{ k}^{n}(R)^{\bot}$.

(2)
 Consider an exact sequence $0\rightarrow K\rightarrow F\rightarrow N^*\rightarrow0$ of left $R$-modules where $F$ is projective. Then, we have the following pushout diagram
$$\xymatrix{
&&0\ar[d]&0\ar[d]&\\
0\ar[r]&K\ar[r]\ar@{=}[d]&F\ar[r]\ar[d]^{\alpha}&N^* \ar[r]\ar[d]&0  \\
0\ar[r]&K\ar[r]& E(F)\ar[r]\ar[d]&L\ar[r]\ar[d]& 0 \\
&&Z\ar@{=}[r]\ar[d]&Z\ar[d]&\\
&&0&0&\\
}$$
where $\alpha$ is an injective envelope of $F$.  By Lemma \ref{1.105}, $N^*\in\mathcal{WI}_{ k}^{n}(R)^{\bot}$. Then similar to the proof of (1), we see 
 ${\rm Ext}_{R}^{1}(Z, N^*)=0$, and hence the sequence $0\rightarrow N^*\rightarrow L\rightarrow Z\rightarrow0$ splits. Thus, $N^*$ is a quotient of $E(F)$ and then we have the exact sequence  $0\rightarrow N^{**}\rightarrow (E(F))^*$ exists with $ (E(F))^*$ is flat. By \cite[Proposition 2.3.5]{JX}, $N$ is pure in $N^{**}$, and hence  $N$ embeds in  flat $R$-module.

  Let  $\beta: N\rightarrow F$ be a  flat preenvelope of $N$. Then, $\beta$ is monic, and so we have the exact sequence $0\rightarrow N\rightarrow F\rightarrow D\rightarrow 0$. By \cite[Lemma 2.1.2]{JX}, $D\in  {^{\bot}\mathscr{F}_{ 0}(R^{op})}\subseteq {^{\bot}\mathcal{WF}_{ 0}^{0}(R^{op})}$, and then by Corollary \ref{1.5qq}, $D$ is in $\mathcal{WI}_{0}^{0}(R)^{\top}$. We show that $D\in\mathcal{WI}_{ k}^{n}(R)^{\top}$. Now, let $K_1$ be in $\mathcal{WI}_{k}^{n}(R)$. Then, there exists an exact sequence $0\rightarrow K_1\rightarrow E(K_1)\rightarrow K_2\rightarrow 0$, where 
$K_2\in\mathcal{WI}_{ k-1}^{n}(R)$ by \cite[Proposition 3.2]{AB}. Then we get the exact sequence 
$$0={\rm Tor}_{2}^{R}(K_2, F)\to {\rm Tor}_{2}^{R}(K_2, D)\to{\rm Tor}_{1}^{R}(K_2, N).$$
We have ${\rm Tor}_{1}^{R}(K_2, N)=0$ and so $ {\rm Tor}_{2}^{R}(K_2, D)=0$. Also,  we have the following exact sequence
$$0={\rm Tor}_{2}^{R}(K_2, D)\to {\rm Tor}_{1}^{R}(K_1, D)\to {\rm Tor}_{1}^{R}(E(K_1), D).$$
Since $D$  is in $\mathcal{WI}_{0}^{0}(R)^{\top}$ and $E(K_1)\in\mathcal{WI}_{0}^{0}(R)$, we deduce that ${\rm Tor}_{1}^{R}(E(K_1), D)=0$. Hence, ${\rm Tor}_{1}^{R}(K_1, D)=0$ and then $D\in\mathcal{WI}_{ k}^{n}(R)^{\top}$.
\end{proof}

\begin{lem}\label{1.103}
	Assume that $ {\rm l.nsp. gldim}(R)<\infty$. Then,  ${\rm l.nsp. gldim}(R)$=${\rm n\text{-}wid}(_RR)$.
\end{lem}
\begin{proof}
	By \cite[Proposition 3.5  ]{AB}, ${\rm l.nsp. gldim}(R)=\sup\{ n$-${\rm wid}_{R}(M) \ \mid \ {\rm M \ is \ a\ left} \ R$-$module \}={\sup}\{ n$-${\rm wfd}_{R}(M) \ \mid \ {\rm M \ is \ a\ right} \ R$-$module \}$. 
	So, we have ${\rm n\text{-}wid}(_RR)\leq {\rm l.nsp. gldim}(R)$. Now, let ${\rm l.nsp. gldim}(R)=k.$ We show that ${\rm n\text{-}wid}(_RR)\geq k.$ For let $U$ be an $n$-super finitely presented left $R$-module with associated special super finitely presented module $K_{n-1}$. So there exists left $R$-module $M$ such that ${\rm Ext}_{R}^{k}(K_{n-1}, M)\neq0$. Consider an exact sequence $0\rightarrow L\rightarrow F\rightarrow M\rightarrow 0$ where $F$ is free. Then we have the following exact sequence:
	$${\rm Ext}_{R}^k(K_{n-1}, F)\rightarrow{\rm Ext}_{R}^k(K_{n-1}, M)\rightarrow {\rm Ext}_{R}^{k+1}(K_{n-1}, L).$$
	We have ${\rm Ext}_{R}^{k+1}(K_{n-1}, L)$ (because ${\rm l.nsp. gldim}(R)=k$). Therefore ${\rm Ext}_{R}^k(K_{n-1},R)\neq0$, otherwise ${\rm Ext}_{R}^k(K_{n-1}, F)=0$ and so from the exact sequence above, ${\rm Ext}_{R}^k(K_{n-1}, M)=0$ which is a contradition. Consequently ${\rm n\text{-}wid}(_RR)\geq k$.
\end{proof}
The following theorem extends the results of Mao and Ding \cite[Theorem 6.4 ]{LM3} and of Selvaraj and Prabakaran \cite[ Theorem 6]{B2B}.
 
\begin{thm}\label{Thm-lnspg-finite}
Let $n$ be a non-negative integer and assume that ${\rm n\text{-}wid}(_RR)\leq k$ where $k\geq 1$. Then,  the following conditions are equivalent:
\begin{enumerate}
\item [\rm (1)] 
${\rm l.n.sp. gldim}(R)<\infty$.
\item [\rm (2)] 
${\rm l.n.sp. gldim}(R)\leq k$.
\item [\rm (3)] 
Every left $R$-module in $^{\bot}\mathcal{WI}_{ k}^{n}(R)$ belongs to $\mathcal{WI}_{ k}^{n}(R)$.
\item [\rm (4)] 
Every left $R$-module in $\mathcal{WI}_{ k-1}^{n}(R)^{\bot}$ is injective.
\item [\rm (5)] 
Every left $R$-module in $\mathcal{WI}_{ k-1}^{n}(R)^{\bot}$ is $n$-weak injective.
\item [\rm (6)] 
Every left $R$-module in $\mathcal{WI}_{ k}^{n}(R)^{\bot}$ is injective.
\item [\rm (7)] 
Every left $R$-module in $\mathcal{WI}_{ k}^{n}(R)^{\bot}$ is $n$-weak injective.
\item [\rm (8)] 
Every left $R$-module in $\mathcal{WI}_{ k}^{n}(R)^{\bot}$ belongs to $\mathcal{WI}_{ k}^{n}(R)$.
\item [\rm (9)] 
Every left $R$-module (with $M\in{^{\bot}\mathcal{WI}_{ k-1}^{n}(R)}$) has a monic $\mathcal{WI}_{ k-1}^{n}(R)$-cover.
\end{enumerate}
If $R$ is a left coherent ring, then the above conditions are also equivalent to:
\begin{enumerate}
\item [\rm (10)] 
Every right $R$-module in $\mathcal{WI}_{ k-1}^{n}(R)^{\top}$ is flat.
\item [\rm (11)] 
Every right $R$-module in $\mathcal{WI}_{ k-1}^{n}(R)^{\top}$ is $n$-weak flat.
\item [\rm (12)] 
Every right $R$-module in $\mathcal{WI}_{k}^{n}(R)^{\top}$ is flat.
\item [\rm (13)] 
Every right $R$-module in $\mathcal{WI}_{ k}^{n}(R)^{\top}$ is $n$-weak flat.
\end{enumerate}
\end{thm}
\begin{proof}
    $(2)\Rightarrow (1)$, $(4)\Rightarrow (6)\Rightarrow (7)$, $(4)\Rightarrow(5)\Rightarrow (7)$, $(10)\Rightarrow(11)\Rightarrow(13)$ and $(10)\Rightarrow(12)\Rightarrow(13)$ are clear.

$(1)\Rightarrow (2)$ Clear by Lemma \ref{1.103}.

$(2)\Leftrightarrow (3)\Leftrightarrow(9)$ and $(2)\Rightarrow (8)$ Follow from \cite[Proposition 4.10]{AB}.

$(8)\Rightarrow (2)$  Let $M$ be a left $R$-module. Since ${\rm n\text{-}wid}(_RR)\leq k$, it follows that $R\in\mathcal{WI}_{ k}^{n}(R)$. Then by  \cite[Proposition 4.8]{AB}, there is a short exact sequence $0\rightarrow K\rightarrow F\stackrel{\displaystyle h}\rightarrow M\rightarrow0$, where  
$h$ is an epic $\mathcal{WI}_{ k}^{n}(R)$-cover of $M$. Hence by \cite[Lemma 2.1.1]{JX} and \cite[Proposition 4.12]{AB}, we deduce that $K\in\mathcal{WI}_{ k}^{n}(R)^{\bot}$. By
assumption, $K\in\mathcal{WI}_{ k}^{n}(R)$, and so \cite[Corollary 3.4 ]{AB} implies that $M\in\mathcal{WI}_{ k}^{n}(R)$. Consequently, by  \cite[Proposition 4.10]{AB}, ${\rm l.n.sp. gldim}(R)\leq k$.

 $(6)\Rightarrow (4)$ Clear by Lemma \ref{1.104}(1).

$(7)\Rightarrow (5)$ Clear by Lemma \ref{1.104}(1) and  \cite[Corollary 3.4(1)]{AB}.

$(2)\Leftrightarrow (6)$ Follows from \cite[Propositions 4.10 and 4.12 (1)]{AB}.

$(5)\Rightarrow (4)$
Let $M \in \mathcal{WI}_{ k-1}^{n}(R)^{\bot}$. By assumption, $M$ is also $n$-weak injective. Now consider an exact sequence $0\rightarrow M\rightarrow E\rightarrow D\rightarrow0$, where $E$ is injective and so also $n$-weak injective. Then by \cite[Corollary 3.4 ]{AB},
$D$ is $n$-weak injective. By \cite[Theorem 2.13(1)]{AB}, ${\rm Ext}_{R}^k(K_{n-1}, D)=0$ for any  special super finitely presented $K_{n-1}$, and hence by 
\cite[Proposition 3.2]{AB}, $D\in \mathcal{WI}_{ k-1}^{n}(R)$. Then, ${\rm Ext}_{R}^{1}(D, M)=0$. So the above exact sequence splits and consequently $M$ is injective.

$(7)\Rightarrow (6)$  Similar to the proof of $(5)\Rightarrow (4)$.

$(12)\Rightarrow (10)$ Trivial by  Lemma \ref{1.104}(2).

$(13)\Rightarrow (11)$ Trivial by Lemma \ref{1.104}(2) and  \cite[Corollary 3.4(2)]{AB}.

$(13)\Rightarrow (12)$
Let $M$ be in $\mathcal{WI}_{ k}^{n}(R)^{\top}$. Then by (13), $M$ is $n$-weak flat. By Lemma \ref{1.105}, $M^*\in\mathcal{WI}_{ k}^{n}(R)^{\bot}$ and is also $n$-weak injective by \cite[Proposition 2.6]{AB}. Similar to the proof of $(5)\Rightarrow (4)$, we see that
$M^*\in\mathcal{WI}_{ k}^{n}(R)^{\bot}$ and is also injective. Therefore by   Lemma \ref{1.105} again, $M\in\mathcal{WI}_{ k}^{n}(R)^{\top}$ and is also flat.

$(12)\Rightarrow (1)$
Every FP-injective $R$-module is $n$-weak injective. So, if $\mathscr{FI_n}_k$ denotes the class of FP-injective right $R$-modules with dimension at most $k$, then ${\mathscr{FI_n}_k}^{\top}\subseteq\mathcal{WI}_{ k}^{n}(R)^{\top}$.
By assumption every right $R$-module in ${\mathscr{FI_n}_k}^{\top}$ is flat. Hence by \cite[Theorem 6.4]{LM3}, $wD(R)\leq k$, and then by Proposition \ref{prop-lnsp-wgldim}, ${\rm l.n.sp. gldim}(R)<\infty$.

$(6)\Rightarrow (12)$ Follows from Lemma \ref{1.103}.
\end{proof}

\begin{thm}\label{Thm-lnspgldim-k}
Let $n$ and $k$ be non-negative integers and consider the following assertions:
\begin{enumerate}
\item [\rm (1)] 
${\rm l.n.sp. gldim}(R)\leq k$.
\item [\rm (2)] 
Every $(n, k)$-weak cotorsion right $R$-module is injective.
\item [\rm (3)] 
${\rm id}_{R}(M)\leq k$ for every $(n,0)$-weak cotorsion right $R$-module $M$.
\item [\rm (4)] 
Every $(n, k)$-weak cotorsion right $R$-module is in $\mathcal{WF}_{ k}^{n}(R^{op})$.
\item [\rm (5)] 
${\rm id}_{R}(M)\leq k^{'}$ for any $k^{'}$ with $0\leq k^{'}\leq k$ and any $(n, k-k^{'})$-weak cotorsion right $R$-module
$M$.
\item [\rm (6)] 
${\rm id}_{R}(M)\leq k^{'}$ for some $k^{'}$ with $0\leq k^{'}\leq k$ and any $(n, k-k^{'})$-weak cotorsion right $R$-modul
$M$.
\item [\rm (7)] 
${\rm Ext}_{R}^{1+i}(M,N)=0$ for any $i\geq m$ and any $(n,k+m)$-weak cotorsion module $M$ and $(n,k)$-weak cotorsion module $N$.
\item [\rm (8)] 
Every right $R$-module $M$ has a $\mathcal{WF}_{ k}^{n}(R^{op})$-cover with the unique mapping property.

Then, $(1)\Leftrightarrow (2)\Leftrightarrow(4)\Leftrightarrow(7)\Leftrightarrow(8)$ and $(1)\Rightarrow (5)\Rightarrow (6)\Rightarrow (3)$.
\end{enumerate}
\end{thm}
\begin{proof}
$(1)\Leftrightarrow (2)$ 
 ${\rm l.n.sp. gldim}(R)\leq k$ if and only if ${\rm pd}_{R}(K_{n-1})\leq k$ for all special super finitely presented module $K_{n-1}$ of any 
$n$-special super finitely presented left $R$-module $U$,  if and only if  every right $R$-module is in  $\mathcal{WF}_{ k}^{n}(R^{op})$  by \cite[Proposition 3.5]{AB}, if and only if 
 ${\rm Ext}_{R}^{1}(N, C)=0$  for any $R$-module $N$ and any $(n,k)$-weak cotorsion right $R$-module $C$.

$(1)\Rightarrow (4)$ Similar to the proof of $(1)\Rightarrow (2)$, it follows that every right $R$-module is in  $\mathcal{WF}_{ k}^{n}(R^{op})$.

$(4)\Rightarrow (1)$
Let $N$ be a right $R$-module. Then by \cite[Proposition 4.13 ]{AB}, there exists an exact sequence $0\rightarrow N\rightarrow X\rightarrow D\rightarrow0$, where $X$ is in 
$\mathcal{WF}_{ k}^{n}(R^{op})^{\bot}=\mathcal{WC}_{ k}^{n}(R^{op})$ and
$D$ is in  $\mathcal{WF}_{ k}^{n}(R^{op})$. By assumption, $X$ is in  $\mathcal{WF}_{ k}^{n}(R^{op})$. Therefore, from \cite[Corollary 34]{AB}, $N\in\mathcal{WF}_{ k}^{n}(R^{op})$, and so by 
\cite[Proposition 3.5]{AB}, it follows that ${\rm pd}_{R}(K_{n-1})\leq k$ for all special super finitely presented module $K_{n-1}$ of any 
$n$-special super finitely presented left $R$-module $U$. Consequently ${\rm l.n.sp. gldim}(R)\leq k$.

$(4)\Rightarrow (7)$
${\rm Ext}_{R}^{1}(M^{'},N)=0$ for all $(n,k)$-weak cotorsion  right $R$-modules $M^{'}$ and $N$ by assumption. Let $N$ be an  $(n,k)$-weak cotorsion right $R$-module and $D^m$
the $m$th cosyzygy of $N$. Then, we have  ${\rm Ext}_{R}^{m+1}(M^{'}, N)\cong{\rm Ext}_{R}^{1}(M^{'}, D^m)$. By Proposition \ref{1.66}(2), $D^m$ is $(n,k+m)$-weak cotorsion, and also it easy to see that
$D^m$ is $(n,k)$-weak cotorsion.  Since $M^{'}\in\mathcal{WF}_{ k}^{n}(R^{op})$, we deduce that ${\rm Ext}_{R}^{1}(M^{'}, D^m)=0$, and hence ${\rm Ext}_{R}^{1+i}(M^{'}, N)=0$ for any $i\geq m$. On the other hand, every 
$(n,k+m)$-weak cotorsion $M$ is $(n,k)$-weak cotorsion. Therefore, ${\rm Ext}_{R}^{1+i}(M,N)=0$ for any $i\geq m$ and for all $(n,k+m)$-weak cotorsion modules $M$ and $(n,k)$-weak cotorsion right $R$-modules $N$.

$(7)\Rightarrow (4)$
By assumption, ${\rm Ext}_{R}^{1}(M,N)=0$ for all $(n,k)$-weak cotorsion  right $R$-modules $M$ and $N$. So every $(n, k)$-weak cotorsion right $R$-module is in $\mathcal{WF}_{ k}^{n}(R^{op})$.

$(1)\Rightarrow (8)$ Clear.

$(8)\Rightarrow (4)$
Let $M$ be an $(n, k)$-weak cotorsion right $R$-module. Then by assumption,  we have the following commutative diagram
\begin{displaymath}\xymatrix{
		&&&0\ar[d]\\
		0\ar[r]&M \ar[r]^{\alpha} \ar[rd]_0& X \ar[r]^{\beta}\ar[d]^{\gamma\beta}& D\ar[r]\ar[ld]^{\gamma}& 0, \\
		&&X^{'}}
\end{displaymath}
where $\alpha$ and $\gamma$  are $\mathcal{WF}_{k}^n(R^{op})$-envelope. Since $\gamma\beta\alpha=0$, it follows that $\gamma\beta=0$ by assumption. Consequently 
$L={\rm im}(\beta)\subseteq {\rm ker}(\gamma)=0$ and so $L=0$. Hence, $M\cong X$, and so $M$ is in  $\mathcal{WF}_{ k}^{n}(R^{op})$.

$(1)\Rightarrow (5)$
Let $N$ be a right $R$-module. By \cite[Proposition 3.3] {AB}, $n.{\rm wfd}_{R}(N)\leq k$, and so there is an exact sequence:
$$ 0\to F_{k}\to F_{k-1}\to\cdots\to F_1\to F_0\to N\to 0,$$
where any $F_i$ is $n$-weak flat. Since $0\leq k^{'}\leq k$, the $(k^{'}-1)$th syzygy of $N\in\mathcal{WF}_{ k-k^{'}}^{n}(R^{op})$. If $L$ is the $(k^{'}-1)$th syzygy of $N$, then
 ${\rm Ext}_{R}^{k^{'}+1}(N, M)\cong{\rm Ext}_{R}^{1}(L, M)=0$ for any   $(n, k-k^{'})$-weak cotorsion right $R$-module $M$. Hence, ${\rm id}_{R}(M)\leq k^{'}$.

$(5)\Rightarrow (6)$ Clear.

$(6)\Rightarrow (3)$ Let $M$ be an $(n,0)$-weak cotorsion right $R$-module. Then, by Proposition \ref{1.66}(2), the $(n, k-k^{'})$th cosyzygy  $D^{ k- k^{'}}$ of $M$ is $(n, k-k^{'})$-weak cotorsion. 
So, ${\rm id}_{R}(D)\leq k^{'}$ and hence ${\rm id}_{R}(M)\leq k$.

\end{proof}

 \begin{prop}\label{prop-lnspgldim-k+2}
Let $n$ and $k$ be non-negative integers. If $R$ satisfies one of the following conditions:
\begin{enumerate}
\item [\rm (1)] 
Every $(n, k)$-weak cotorsion right $R$-module has a $\mathcal{WF}_{ k}^{n}(R^{op})$-envelope with the unique mapping property.
\item [\rm (2)] 
Every finitely presented right $R$-module has a $\mathcal{WF}_{ k}^{n}(R^{op})$-envelope with the unique mapping property.
\end{enumerate}
Then, ${\rm l.n.sp. gldim}(R)\leq k+2$.
\end{prop}
\begin{proof}
Assume (1). Then, by \cite[Theorem 4.5]{AB}, every right $R$-module has a $\mathcal{WF}_{ k}^{n}(R^{op})$-cover. So, by Proposition \ref{1.5}(3) and \cite[Lemma 2.1.1]{JX},
 we have exact sequences
 $0\to K\stackrel{\displaystyle i}\to F\stackrel{\displaystyle \alpha}\to N\to0$  and 
 $0\to K^{'}\stackrel{\displaystyle i^{'}}\to F^{'}\stackrel{\displaystyle \alpha^{'}}\to K\to0,$
where $\alpha$ and $\alpha^{'}$ are $\mathcal{WF}_{ k}^{n}(R^{op})$-covers of $N$ and $K$, respectively. Also $K$ and $K^{'}$ are $(n, k)$-weak cotorsion. Hence, we obtain the exact sequence:
$$0\to K^{'}\stackrel{\displaystyle i^{'}}\to F^{'}\stackrel{\displaystyle i{\alpha}^{'}}\to F\stackrel{\displaystyle \alpha}\to N\to0.$$
Let $\beta : K^{'}\to L$ be a $\mathcal{WF}_{ k}^{n}(R^{op})$-envelope with the unique mapping property. Then there exists $\gamma : L\to F^{'}$ such that $i^{'}=\gamma\beta$. So $ i{\alpha}^{'}\gamma\beta= i{\alpha}^{'}i^{'}=0$, and then $i{\alpha}^{'}\gamma=0$, which
implies that ${\rm im}(\gamma)\subseteq{\rm ker}( i{\alpha}^{'})={\rm im}(i^{'})$. Therefore, there exists $\delta : L\to K^{'}$ such that $i^{'}\delta=\gamma$. Thus, we get the following exact commutative diagram:
\begin{displaymath}\xymatrix{
&&&&&&\\
0\ar[r]&K^{'}\ar[r]^{i^{'}}\ar[rd]^{\beta}&F^{'} \ar[r]^{i\alpha^{'}}& F \ar[r]^{\alpha}&N\ar[r]& 0, \\
&L\ar[u]^{\delta}\ar@{=}[r]&L\ar[u]^{\gamma}&&&}
\end{displaymath}
where $i^{'}\delta\beta=i^{'}$. So $\delta\beta=1_{K^{'}}$ (because $i^{'}$ is monic). Thus, $K^{'}$ is isomorphic to
a direct summand of $L$, and hence by using of \cite[Proposition 2.9 ]{AB}, it follows that 
$K^{'}\in\mathcal{WF}_{k}^{n}(R^{op})$. Therefore, we deduce that $K\in\mathcal{WF}_{ k+1}^{n}(R^{op})$ and consequently $N\in\mathcal{WF}_{ k+2}^{n}(R^{op})$.  Hence, by \cite[Proposition 3.3]{AB}, $n$-${\rm wfd}_{R}(N)\leq k+2$, and so  \cite[Proposition 3.5]{AB} implies that ${\rm l.n.sp. gldim}_{R}(R)\leq k+2$.

For the converse, by \cite[Lemma 3.2]{Ding}, every right $R$-module has a $\mathcal{WF}_{ k+1}^{n}(R^{op})$-envelope with the unique
mapping property (because $\mathcal{WF}_{ k+1}^{n}(R^{op})$ is closed under direct limits).
\end{proof}


\end{document}